\documentclass[11pt]{amsart}

\usepackage[T1]{fontenc}
\usepackage[utf8]{inputenc}
\usepackage{lmodern}
\usepackage{amsmath,amssymb,amsthm,mathtools,mathrsfs}

\mathtoolsset{showonlyrefs}
\numberwithin{equation}{section}
\usepackage{booktabs,array}
\usepackage[a4paper,margin=1.15in]{geometry}
\usepackage{microtype}
\usepackage[colorlinks=true,linkcolor=blue,citecolor=blue,urlcolor=blue]{hyperref}

\newtheorem{theorem}{Theorem}[section]
\newtheorem{proposition}[theorem]{Proposition}
\newtheorem{lemma}[theorem]{Lemma}

\theoremstyle{remark}
\newtheorem{remark}[theorem]{Remark}

\newcommand{\R}{\mathbb R}
\newcommand{\C}{\mathbb C}
\newcommand{\PP}{\mathbb P}
\newcommand{\aff}{\operatorname{aff}}
\DeclareMathOperator{\ord}{ord}

\title[The finiteness conjecture for equilibria of electric fields]{The finiteness conjecture for equilibria of electric fields generated by point charges of one sign}

\author{Alberto Enciso}
\address{Instituto de Ciencias Matem\'aticas, Consejo Superior de Investigaciones Cient\'ificas, 28049 Madrid, Spain}
\email{aenciso@icmat.es}

\author{Daniel Peralta-Salas}
\address{Instituto de Ciencias Matem\'aticas, Consejo Superior de Investigaciones Cient\'ificas, 28049 Madrid, Spain}
\email{dperalta@icmat.es}

\begin{document}

\begin{abstract}
We prove that the electric field generated in three-dimensional space by
finitely many point charges of one sign has only finitely many equilibrium
points, thereby answering a 1969 question of Morse and Cairns (restated by Eremenko in 2008 and, as a conjecture, by Shapiro in 2015).  More
generally, for nonzero charges $q_i\in\R$ of possibly mixed signs at
distinct sites $\mathbf a_i\in\R^3$, we show that the Coulomb field has at
most
$
 2^{N-4}(N-1)(9N^2+9N+10)
$
equilibria in the region where
$S(\mathbf x):=\sum_iq_i|\mathbf x-\mathbf a_i|^{-3}$ does not vanish.
Of course, for charges of one
sign, $S$ is nonzero everywhere.  The proof rules out
curves of equilibria using algebraic geometry and complex analysis on an
associated complex curve, and then applies a B\'ezout count to obtain a
quantitative bound.  Well-known examples show that, in the mixed-sign case,
the Coulomb field can vanish on curves contained in the zero set of $S$.
\end{abstract}

\maketitle

\section{Introduction}\label{sec:introduction}

Let $N\ge1$, let $\mathbf a_1,\ldots,\mathbf a_N\in\R^3$ be pairwise
distinct, and let $q_1,\ldots,q_N\in\R\setminus\{0\}$.  The Coulomb
potential and the electric field generated by these point charges are, respectively,
\begin{equation}\label{eq:intro-U-E}
 U(\mathbf x):=\sum_{i=1}^N\frac{q_i}{|\mathbf x-\mathbf a_i|},
 \qquad
 E(\mathbf x):=-\nabla U(\mathbf x)
 =\sum_{i=1}^Nq_i\frac{\mathbf x-\mathbf a_i}
 {|\mathbf x-\mathbf a_i|^3}.
\end{equation}
They are defined on the set
\(
 \Omega:=\R^3\setminus\{\mathbf a_1,\ldots,\mathbf a_N\},
\)
and the electric equilibria are the critical points of the potential, i.e.,
\[
 \operatorname{Crit}(U):=
 \{\mathbf x\in\Omega:E(\mathbf x)=0\}.
\]

Our understanding of this set is still rather limited.  Already in the
nineteenth century, Maxwell conjectured~\cite{Maxwell1892} that there
should not be more than $(N-1)^2$ nondegenerate equilibria.  This was recently disproved by Arathoon, Ball and Kvalheim,
who found a five-charge configuration with at least twenty-four nondegenerate equilibrium
points~\cite{ArathoonBallKvalheim2026}.

When charges of different signs are allowed, there are elementary examples,
recalled in Appendix~\ref{sec:mixed-sign}, in which $E$ vanishes on a line
or a circle.  It has nevertheless long been conjectured that $E$ cannot
vanish on a curve (and therefore, all the critical points are isolated) when all the charges have the same sign.  To the best of
our knowledge, the question was first recorded by Morse and Cairns in
1969~\cite[Section IV.32]{MorseCairns1969} (see also the collection of unsolved problems of Eremenko, who stated the question of the same-sign finiteness in 2008~\cite{Ere}).  Shapiro later stated same-sign finiteness as
a conjecture~\cite{Shapiro2015}; Bauer, Catanese and Di Scala recorded the
Coulomb analytic-arc problem as open~\cite{BauerCataneseDiScala2017}; and
Zolotov emphasized that bounds for isolated equilibria do not exclude
nonisolated ones~\cite{Zolotov2023}.

Our objective is to prove the same-sign finiteness conjecture.  In fact,
we obtain an explicit unconditional bound for the number of equilibria,
which grows exponentially with the number of charges.

A key idea is not to impose a sign from the start.  Instead, we introduce the function
\begin{equation}\label{eq:intro-S}
 S(\mathbf x):=\sum_{i=1}^N
 \frac{q_i}{|\mathbf x-\mathbf a_i|^3}
\end{equation}
and consider the nonexceptional equilibrium set
\[
 Z:=\{\mathbf x\in\Omega:E(\mathbf x)=0,\ S(\mathbf x)\ne0\}.
\]
When all the charges have one sign, $S$ does not vanish and
$Z=\operatorname{Crit}(U)$.

\begin{theorem}\label{thm:main}
For any nonzero charges $q_1,\ldots,q_N\in\R$ and any pairwise distinct
sites $\mathbf a_1,\ldots,\mathbf a_N\in\R^3$, the set $Z$ is finite and
\begin{equation}\label{eq:main-bound}
 \#Z\le 2^{N-4}(N-1)(9N^2+9N+10).
\end{equation}
In particular, the same-sign finiteness conjecture holds.
\end{theorem}

The bound~\eqref{eq:main-bound} requires no genericity or nondegeneracy assumptions.  For
nondegenerate configurations, Gabrielov, Novikov and Shapiro obtained
fewnomial bounds~\cite{GabrielovNovikovShapiro2007}, while Zolotov proved
an unconditional bound for isolated critical points
\cite[Theorem~2(4)]{Zolotov2023}; we observe that for same-sign configurations our bound is sharper than Zolotov's.  Edelsbrunner, Fillmore and Oliveira
obtained the estimate $2^N(3N-2)^3$ when all data except one site are fixed
and that site is chosen generically
\cite[Result~1]{EdelsbrunnerFillmoreOliveira2026}.  A recent preprint gives
an upper bound of six for three positive charges in the nondegenerate setting
\cite{GabrielovNovikovNovikovShapiro2026}.

Let us briefly describe the main idea of our proof.  Positivity of charges is difficult to use directly
to rule out a curve of equilibria, so we allow arbitrary signs and work on
the equilibrium set where $S$ is nonzero, which why we introduce the set~$Z$. The analysis of~$Z$ has two parts.  First, if there were infinitely many
nonexceptional equilibria, elementary real algebraic geometry would provide
a nonconstant real-analytic curve of such points.  We place the position
and distance functions on a compact complex curve and compare their
behavior at finite points and at infinity.  Certain meromorphic
combinations then have no poles and must be constant.  The equilibrium
identities force the position itself to be constant, which is a
contradiction.  This is carried out in
Sections~\ref{sec:semialgebraic}--\ref{sec:qualitative}.

Second, each nonexceptional equilibrium is encoded as a solution of a
compact homogeneous system.  The first part shows that the corresponding
complex solution is isolated.  A global intersection count, together with
the explicit contributions of the charge sites and the points lying over
the distinguished point at infinity, then gives the bound in
Section~\ref{sec:quantitative}.

\section{From infinitely many equilibria to a compact connected Riemann surface}
\label{sec:semialgebraic}

In this section we pass from the assumption of the existence of an infinite real equilibrium set to a
complex curve.  The idea is easy: if $Z$ were infinite, we could study
a nonconstant one-parameter family of its points using meromorphic
functions on a compact Riemann surface.  The charges may have arbitrary
signs throughout.  Obviously, the case $N=1$ has no equilibria, so we assume $N\ge2$.

For a physical point $\mathbf x\in\Omega$, set
\[
 \rho_i(\mathbf x):=|\mathbf x-\mathbf a_i|.
\]
The symbols $x=(x_1,x_2,x_3)$ and $r_1,\ldots,r_N$ introduced in
Lemma~\ref{lem:curve-reduction} will instead denote meromorphic functions
on a Riemann surface.  On the lifted real arc produced there, they agree
with the physical position and the positive distances $\rho_i$; elsewhere
they may take complex values or have poles.

With $S$ as in~\eqref{eq:intro-S}, define
\[
 T(\mathbf x):=\sum_{i=1}^N
 \frac{q_i\mathbf a_i}{\rho_i(\mathbf x)^3}.
\]
The electric field can be written as
\begin{equation}\label{eq:physical-field-decomposition}
 E(\mathbf x)
 =\sum_{i=1}^Nq_i\frac{\mathbf x-\mathbf a_i}
 {\rho_i(\mathbf x)^3}
 =S(\mathbf x)\mathbf x-T(\mathbf x).
\end{equation}
Thus, by~\eqref{eq:physical-field-decomposition}, $E(\mathbf x)=0$ is
equivalent to the balance relation
$S(\mathbf x)\mathbf x=T(\mathbf x)$.  The construction below turns this
relation along a real curve of equilibria into an identity of meromorphic
functions on a compact Riemann surface.

There is no canonical single-valued holomorphic function obtained by
replacing $\mathbf x$ with a complex variable in
$|\mathbf x-\mathbf a_i|$.  What does complexify canonically is the
polynomial relation satisfied by the distance.  We therefore introduce the
distances as additional coordinates.  For $\mathbf x\in\R^3$ and
$\widetilde\rho=(\widetilde\rho_1,\ldots,\widetilde\rho_N)\in\R^N$,
set
\[
 D(\widetilde\rho):=\prod_{i=1}^N\widetilde\rho_i^3
\]
and define the polynomial maps
\begin{align*}
 P(\mathbf x,\widetilde\rho)
 &:=\sum_{i=1}^Nq_i(\mathbf x-\mathbf a_i)
       \prod_{j\ne i}\widetilde\rho_j^3,\\
 Q(\widetilde\rho)
 &:=\sum_{i=1}^Nq_i\prod_{j\ne i}\widetilde\rho_j^3.
\end{align*}

On the graph of the physical distances, i.e., on $\widetilde\rho_i=|\mathbf x-\mathbf a_i|$, $P$ and $Q$ are the numerators of
$E$ and $S$ after clearing their common denominator.  Define
\begin{equation}\label{eq:lifted-Z}
 \widehat Z:=
 \left\{(\mathbf x,\widetilde\rho)\in\R^{3+N}:
 \begin{array}{l}
 \widetilde\rho_i>0,\quad
 \widetilde\rho_i^2=(\mathbf x-\mathbf a_i)\cdot
                    (\mathbf x-\mathbf a_i)\quad(1\le i\le N),\\
 P(\mathbf x,\widetilde\rho)=0,\quad
 Q(\widetilde\rho)\ne0
 \end{array}
 \right\}.
\end{equation}
This is a semialgebraic set.  The inequalities
$\widetilde\rho_i>0$ select the physical square roots, so
\[
 \widetilde\rho_i=\rho_i(\mathbf x)=|\mathbf x-\mathbf a_i|
\]
at every point of $\widehat Z$.  In particular,
$D(\widetilde\rho)>0$ there and
\[
 P(\mathbf x,\widetilde\rho)
 =D(\widetilde\rho)E(\mathbf x),
 \qquad
 Q(\widetilde\rho)
 =D(\widetilde\rho)S(\mathbf x).
\]
Consequently, $\widehat Z$ is the graph over $Z$ of the physical distance
map.  Projection onto the first three coordinates is a bijection
$\widehat Z\to Z$, with inverse
\[
 \mathbf x\longmapsto
 \bigl(\mathbf x,\rho_1(\mathbf x),\ldots,\rho_N(\mathbf x)\bigr).
\]
In the statement of the following lemma, $I$ denotes a real interval, which can be identified with $(0,1)$ with no loss of generality.
\begin{lemma}\label{lem:curve-reduction}
If $Z$ is infinite, there are a compact connected Riemann surface $C$,
meromorphic functions
\[
 x=(x_1,x_2,x_3),\quad r_1,\ldots,r_N
\]
on $C$, a nonconstant real-analytic arc $\mathbf x:I\to Z$, and a
real-analytic map $\xi:I\to C$ such that
\begin{equation}\label{eq:arc-dictionary}
 x(\xi(t))=\mathbf x(t),
 \qquad
 r_i(\xi(t))=|\mathbf x(t)-\mathbf a_i|>0
 \quad(t\in I).
\end{equation}
The meromorphic map $x$ is nonconstant, no $r_i$ is identically zero, and
\begin{align}
 r_i^2&=(x-\mathbf a_i)\cdot(x-\mathbf a_i),
 &&1\le i\le N,\label{eq:distance-curve}\\
 \sum_{i=1}^Nq_i\frac{x-\mathbf a_i}{r_i^3}&=0,
 \label{eq:field-curve}
\end{align}
on $C$. Moreover, $\sum_iq_i/r_i^3$ is not identically zero on $C$.
\end{lemma}

\begin{proof}
An elementary argument of unique continuation shows that the critical set $\operatorname{Crit}(U)$ has dimension at most~1 (because $U$ is a harmonic function in $\Omega$).
Since $Z$ is infinite and $\widehat Z\to Z$ is a bijection,
$\widehat Z$ is infinite, and has dimension $\leq 1$.  It is standard that an infinite semialgebraic set has positive
dimension, so a Nash stratification of $\widehat Z$, cf.~\cite[Proposition 9.1.8]{BochnakCosteRoy1998}, taking into account the aforementioned dimensional bound, provides a nonconstant
real-analytic, connected semialgebraic arc
\[
 \gamma(t)=
 \bigl(\mathbf x(t),\widetilde\rho_1(t),\ldots,
       \widetilde\rho_N(t)\bigr)\in\widehat Z.
\]
Such an arc is usually called a Nash arc; see
\cite[Section~8]{BochnakCosteRoy1998}.  Its $\mathbf x$
component is nonconstant: if it were constant, the distance
equations would make every $\widetilde\rho_i$ constant as well.

Let $\Gamma\subset\C^{3+N}$ be the (complex) Zariski closure of $\gamma(I)$,
that is, the smallest complex algebraic set containing the arc; since $\gamma(I)$ is connected, $\Gamma$ is also connected by minimality.  Each
coordinate function of a Nash arc is algebraic over the space of rational functions with real coefficients~$\R(t)$, and hence over
$\C(t)$.  Therefore the field generated by the coordinates has
transcendence degree at most one, so $\dim_\C\Gamma\le1$.  Since
$\mathbf x(t)$ is nonconstant, $\Gamma$ has complex dimension exactly one because it is connected.
Moreover, $\Gamma$ is irreducible.  Indeed, if polynomials $F$ and $G$
satisfy $(FG)\circ\gamma=0$, then the real-analytic functions
$F\circ\gamma$ and $G\circ\gamma$ have product zero on an interval, and
one of them must vanish identically.  Thus the ideal of polynomials
vanishing on the arc is prime.

Let $\overline\Gamma\subset\PP^{3+N}$ be the projective closure of
$\Gamma$, and let
\[
 \nu:C\longrightarrow\overline\Gamma
\]
be its normalization.  Passing to $\overline\Gamma$ adds finitely many
points at infinity.  Normalization separates the local branches through
each singular point: every branch is represented by its own smooth point
of $C$, with a local parameter there.  Since $\overline\Gamma$ is an
irreducible projective curve, $C$ is a compact connected Riemann surface.

The singular points of $\Gamma$ are finite.  After shortening $I$, the arc
lies in the smooth locus, where $\nu$ is an isomorphism, and hence has a
real-analytic lift $\xi:I\to C$ satisfying $\nu(\xi(t))=\gamma(t)$.  The
affine coordinate functions on $\Gamma$ pull back to meromorphic functions
on~$C$; denote them by
\[
 (x,r)=(x_1,x_2,x_3;r_1,\ldots,r_N).
\]
By construction,
\[
 x(\xi(t))=\mathbf x(t),
 \qquad
 r_i(\xi(t))=\widetilde\rho_i(t)
 =\rho_i(\mathbf x(t))
 =|\mathbf x(t)-\mathbf a_i|>0.
\]
Thus $x$ is nonconstant and no $r_i$ is identically zero.

Only polynomial {\em equalities}\/ are continued to~$C$.  The positivity and
nonvanishing {\em inequalities}\/ in~\eqref{eq:lifted-Z}, which identify the
physical branch on the real arc, do not hold on $C$.  Every
polynomial equality satisfied on the arc vanishes on $\Gamma$ and hence,
after pullback by $\nu$, holds meromorphically on $C$.  The distance
equations give~\eqref{eq:distance-curve}.  The polynomial field equation (associated to the zeros of the electric field) gives
\[
 0=P(x,r)
 =\left(\prod_{i=1}^Nr_i^3\right)
  \sum_{i=1}^Nq_i\frac{x-\mathbf a_i}{r_i^3}.
\]
The product is not the zero function, so it can be divided out
in the field of meromorphic functions on $C$.  This proves
\eqref{eq:field-curve}.  Finally, along the lifted real arc,
\[
 \sum_{i=1}^N\frac{q_i}{r_i(p(t))^3}
 =S(\mathbf x(t))\ne0
\]
because $\mathbf x(t)\in Z$.  Hence this meromorphic function is not
identically zero on $C$.  It may, of course, have isolated zeros away from
the real arc.
\end{proof}

The construction gives the formula in~\eqref{eq:arc-dictionary};
more explicitly,
\[
 \nu(\xi(t))=\gamma(t)
 =\bigl(\mathbf x(t),\widetilde\rho_1(t),\ldots,
        \widetilde\rho_N(t)\bigr),
 \qquad
 x\circ \xi=\mathbf x,
 \qquad
 r_i\circ \xi=\widetilde\rho_i.
\]
On the lifted real arc, $x(\xi(t))$ is a point of $\R^3$ and $r_i(\xi(t))$ is
its positive distance to $\mathbf a_i$.  Away from that arc, the $x_j$ and
$r_i$ may be complex or have poles.  In particular, $r_i$ is no longer an
absolute value.  What survives on all of $C$ is the polynomial identity
\eqref{eq:distance-curve}, where the dot product is the complex bilinear
extension of the Euclidean product, not the Hermitian product.

Let $\mathcal M(C)$ denote the field of meromorphic functions on $C$.
From now on, all identities are understood in $\mathcal M(C)$, or in its
vector-valued analogue.  We use the same letters $S$ and $T$ for the
meromorphic counterparts of the physical quantities:
\[
 S:=\sum_{i=1}^N\frac{q_i}{r_i^3},
 \qquad
 T:=\sum_{i=1}^N\frac{q_i\mathbf a_i}{r_i^3}.
\]
Expanding~\eqref{eq:field-curve} gives the key identity
\begin{equation}\label{eq:meromorphic-SxT}
 Sx=T.
\end{equation}
The function $S$ is not identically zero, although it may vanish at
isolated points.  Thus division by $S$ is legitimate in $\mathcal M(C)$,
but not as a pointwise operation at one of its zeros.

Differentiating~\eqref{eq:distance-curve} gives
\[
 r_i\,dr_i=(x-\mathbf a_i)\cdot dx.
\]
Using~\eqref{eq:field-curve}, we obtain
\[
 d\!\left(\sum_{i=1}^N\frac{q_i}{r_i}\right)
 =-\sum_{i=1}^Nq_i\frac{(x-\mathbf a_i)\cdot dx}{r_i^3}=0.
\]
A meromorphic function with zero differential on a connected Riemann
surface is constant.  Thus, for some $c\in\C$,
\begin{equation}\label{eq:constant-potential}
 \sum_{i=1}^N\frac{q_i}{r_i}=c
\end{equation}
on $C$.

Finally, let
\[
 A:=\aff_{\R}\{\mathbf a_1,\ldots,\mathbf a_N\}
   =\mathbf a_1+V,
 \qquad
 V_{\C}:=V\otimes_{\R}\C,
 \qquad
 A_{\C}:=\mathbf a_1+V_{\C}.
\]
Thus~$A$ is the affine space defined by the sites, $V$~is the corresponding real vector space, and $V_\C$ is its complexification. For any normal vector $\eta\in V^\perp$,
\[
 \eta\cdot(T-\mathbf a_1S)
 =\sum_{i=1}^N
   \frac{q_i\,\eta\cdot(\mathbf a_i-\mathbf a_1)}{r_i^3}=0.
\]
Together with~\eqref{eq:meromorphic-SxT}, this gives
$S\,\eta\cdot(x-\mathbf a_1)=0$.  Since $S$ is not identically zero,
division in $\mathcal M(C)$ yields
\begin{equation}\label{eq:x-in-affine-span}
 \eta\cdot(x-\mathbf a_1)=0
 \qquad\text{for every }\eta\in V^\perp.
\end{equation}
These are the affine linear equations defining $A_{\C}$.  Thus, wherever
$x$ is finite, it lies in $A_{\C}$, and the defining identities hold
meromorphically on all of $C$. Note that, in the typical case where~$V$ has dimension three, $V^\perp =\{0\}$ and~\eqref{eq:x-in-affine-span} is trivial.

\section{Finite points and points at infinity}
\label{sec:compactification}

At a point $p\in C$, the meromorphic position map $x$ may have a finite
value or a pole.  To treat both possibilities in one space, we use the
paraboloid embedding
\[
 z\longmapsto[1:z:z\cdot z].
\]
Its projective closure adds the possible limits of $x$ at its poles.  The
main advantage is that every complexified squared-distance polynomial
$(z-\mathbf a)\cdot(z-\mathbf a)$ becomes the restriction of a linear
form.

Let $A=\mathbf a_1+V$ and $A_{\C}=\mathbf a_1+V_{\C}$ as above.  For
$z\in A_{\C}$, set
\[
 \iota(z):=[1:z:z\cdot z]\in\PP^4,
\]
where the homogeneous coordinates are written as $[X_0:X:X_4]$, with
$X=(X_1,X_2,X_3)$.  Define the quadric
\begin{equation}\label{eq:YA}
 Y_A:=\left\{[X_0:X:X_4]\in\PP^4:
 \begin{array}{l}
 X\cdot X=X_0X_4,\\
 \eta\cdot(X-\mathbf a_1X_0)=0
 \quad\text{for every }\eta\in V^\perp
 \end{array}
 \right\}.
\end{equation}
On the chart $X_0\ne0$, division by $X_0$ gives
\[
 z:=X/X_0\in A_{\C},
 \qquad
 X_4/X_0=z\cdot z,
\]
so
\[
 Y_A\cap\{X_0\ne0\}=\iota(A_{\C}).
\]
Writing
\[
 W:=X-\mathbf a_1X_0\in V_{\C},
 \qquad
 \widetilde X_4:=X_4-|\mathbf a_1|^2X_0-2\mathbf a_1\cdot W,
\]
the quadric equation becomes $W\cdot W=X_0\widetilde X_4$.  The quadratic
form $W\cdot W-X_0\widetilde X_4$ is nondegenerate on
$V_{\C}\oplus\C^2$.  Hence $Y_A$ is a smooth projective quadric of (complex)
dimension $\dim V$, and it is precisely the projective closure of
$\iota(A_{\C})$.

For each index~$i\in\{1,\dots, N\}$, introduce the linear form
\[
 L_i(X_0,X,X_4)
 :=X_4-2\mathbf a_i\cdot X+|\mathbf a_i|^2X_0,
\]
so that
\[
 L_i(1,z,z\cdot z)
 =(z-\mathbf a_i)\cdot(z-\mathbf a_i).
\]

\begin{lemma}\label{lem:no-common-distance-zero}
Assume that $N\geq 2$.  Then
$L_1,\ldots,L_N$ do not vanish simultaneously at
any point of $Y_A$.
\end{lemma}

\begin{proof}
Suppose first that a common zero lies in the chart $X_0\ne0$.  It can be
written as $\iota(z_0)$, with
\[
 z_0=\mathbf a_1+w,
 \qquad
 w\in V_{\C}.
\]
Set $\mathbf b_j:=\mathbf a_j-\mathbf a_1$.  The equation for $j=1$ gives
$w\cdot w=0$, while the remaining equations give
\[
 2\mathbf b_j\cdot w=|\mathbf b_j|^2
 \qquad(1\le j\le N).
\]
Write $w=u+iv$ with $u,v\in V$.  Taking imaginary parts shows that
$\mathbf b_j\cdot v=0$ for every $j$.  Since the vectors $\mathbf b_j$
span $V$, this forces $v=0$.  Hence $w\in V$, and $w\cdot w=|w|^2=0$
gives $w=0$, contradicting the equation above for any nonzero site
difference.

Now suppose that a common zero lies at infinity, say
$[0:\Xi:\Theta]\in Y_A$.  Then $\Xi\in V_{\C}$, and the equations
$L_j=0$ imply
\[
 \Theta=2\mathbf a_j\cdot\Xi
 \qquad(1\le j\le N).
\]
Subtracting the equation with $j=1$ gives
$(\mathbf a_j-\mathbf a_1)\cdot\Xi=0$ for every $j$.  The site
differences span $V$, and the Euclidean product is nondegenerate on
$V_{\C}$, so $\Xi=0$.  The preceding equations then give $\Theta=0$,
which is impossible in the projective space.
\end{proof}

The function $\Phi:=[1:x:x\cdot x]$ also has a well-defined projective limit
at a pole of $x$.  Near any $p\in C$, its entries are meromorphic;
multiplying by a suitable power of a local coordinate makes them
holomorphic and not all zero at $p$.  Their projective class is independent
of this rescaling and defines a holomorphic map $C\to\PP^4$.  On the dense
open set where $x$ is finite, the map takes values in $Y_A$ by the quadric identity and~\eqref{eq:x-in-affine-span}; since $Y_A$ is closed, its extension does so
everywhere.  We therefore obtain a holomorphic map
\[
 \Phi:C\longrightarrow Y_A.
\]

Fix $p\in C$.  On a neighborhood of $p$, choose holomorphic homogeneous
coordinates
\[
 H=(X_0,X,X_4)
\]
representing $\Phi$, with no common zero.  Any two such choices differ by
a holomorphic unit, so the orders below do not depend on the choice.
Where $x$ is finite, the definition $\Phi=[1:x:x\cdot x]$ gives
\[
 X=X_0x,
 \qquad
 X_4=X_0x\cdot x.
\]
These are meromorphic identities near $p$.  Combining them with
\eqref{eq:distance-curve} yields
\begin{equation}\label{eq:distance-form-identity}
 L_i(H)=X_0r_i^2,
 \qquad 1\le i\le N,
\end{equation}
which holds on all $C$.
Let $\ord_p(f)$ denote the order of a nonzero meromorphic function $f$ at
$p$: it is positive at a zero, negative at a pole, and zero for a local
unit.  Set $\ord_p(0)=+\infty$.  Define
\[
 m_i:=\ord_p(r_i),
 \qquad
 m:=\min_i m_i,
\]
and split the indices as
\[
 I:=\{i:m_i=m\},
 \qquad
 J:=\{i:m_i>m\}.
\]

\begin{proposition}\label{prop:four-cases}
With the notation above,
\begin{equation}\label{eq:order-formulas}
 \ord_p(X_0)=-2m,
 \qquad
 \ord_p\bigl(L_i(H)\bigr)=2(m_i-m).
\end{equation}
Consequently,
\[
 i\in I\iff L_i(\Phi(p))\ne0,
 \qquad
 i\in J\iff L_i(\Phi(p))=0.
\]
Moreover, exactly one of the following four cases occurs:
\begin{enumerate}
\item \emph{Finite real point.}  $X_0(p)\ne0$, $m=0$, and
\[
 \Phi(p)=[1:\mathbf x_0:\mathbf x_0\cdot\mathbf x_0]
\]
for some $\mathbf x_0\in A$.
\item \emph{Finite nonreal point.}  $X_0(p)\ne0$, $m=0$, and
\[
 \Phi(p)=[1:z_0:z_0\cdot z_0]
\]
for some $z_0\in A_{\C}\setminus A$.
\item \emph{The distinguished point at infinity.}  $X_0(p)=0$, $m<0$,
and $\Phi(p)=[0:0:1]$.
\item \emph{Any other point at infinity.}  $X_0(p)=0$, $m<0$, and
\[
 \Phi(p)=[0:\Xi:\Theta],
 \qquad
 \Xi\ne0,
 \qquad
 \Xi\cdot\Xi=0.
\]
\end{enumerate}
In either infinite case, at least one distance function $r_i$ has a pole
at $p$.
\end{proposition}

\begin{proof}
Each $L_i(H)$ is holomorphic near $p$.  By
Lemma~\ref{lem:no-common-distance-zero}, at least one is nonzero at $p$;
hence
\[
 \min_i\ord_p\bigl(L_i(H)\bigr)=0.
\]
Taking orders in~\eqref{eq:distance-form-identity} gives
\[
 \ord_p\bigl(L_i(H)\bigr)
 =\ord_p(X_0)+2m_i.
\]
Taking the minimum over $i$ proves $\ord_p(X_0)=-2m$, and substitution
gives the second formula in~\eqref{eq:order-formulas}.  The description of
$I$ and $J$ follows immediately.

If $X_0(p)\ne0$, then $\ord_p(X_0)=0$, so $m=0$.  The quotient
$x=X/X_0$ is holomorphic at $p$; write its value as $z_0\in A_{\C}$.
Then
\[
 \Phi(p)=[1:z_0:z_0\cdot z_0].
\]
If $z_0\in A$, denote it by $\mathbf x_0$ and obtain Case~1; otherwise
we obtain Case~2.

If $X_0(p)=0$, then $\ord_p(X_0)>0$, so $m<0$.  Write
$\Phi(p)=[0:\Xi:\Theta]$.  The quadric equation gives
$\Xi\cdot\Xi=0$.  If $\Xi=0$, then $\Theta\ne0$ and the point is
$[0:0:1]$ after rescaling.  Otherwise we obtain Case~4.  Finally, $m<0$
means that some $m_i$ is negative, so at least one $r_i$ has a pole.
\end{proof}

\section{Controlling poles and forcing the position map to be constant}
\label{sec:qualitative}

At each point $p\in C$,
Proposition~\ref{prop:four-cases} gives exactly four possibilities: a
finite real point, a finite nonreal point, the distinguished point at
infinity, or another point at infinity.  In this section we first treat these four cases 
and prove the same local estimate in each of them, and then show that this local estimate forces the function~$x$ to be constant. This implies that the set~$Z$ must be finite.

The estimate concerns
\[
 S=\sum_{i=1}^N\frac{q_i}{r_i^3}.
\]
If $m:=\min_i\ord_p(r_i)$, we will prove
\[
 \ord_p(S)\ge-3m.
\]
When $m<0$, the distance functions of smallest order have poles, and the
estimate says that $S$ vanishes to order at least $-3m$.  It will imply that
every product $r_i^3S$ has no poles anywhere on $C$.

Choose a local coordinate $t$ at $p$, with $t(p)=0$.  For a
vector-valued meromorphic function $w=(w_1,w_2,w_3)$, set
\[
 \ord_p(w):=\min_{1\le\ell\le3}\ord_p(w_\ell).
\]

\begin{lemma}\label{lem:nonisotropic-leading}
Suppose that $x$ has a pole at $p$ and can be written as
\[
 x=t^{-h}(w+O(t)),
 \qquad
 h>0,
 \qquad
 w\in\C^3\setminus\{0\}.
\]
If $w\cdot w\ne0$, then $\ord_p(r_i)=-h$ for every $i$.
\end{lemma}

\begin{proof}
For each fixed site $\mathbf a_i$,
\[
 x-\mathbf a_i=t^{-h}(w+O(t)).
\]
Hence
\[
 r_i^2=(x-\mathbf a_i)\cdot(x-\mathbf a_i)
 =t^{-2h}\bigl(w\cdot w+O(t)\bigr).
\]
Since $w\cdot w\ne0$, the right-hand side has order $-2h$.  Therefore
$2\ord_p(r_i)=-2h$.
\end{proof}

\begin{proposition}\label{prop:pole-estimate}
For every $p\in C$,
\[
 \ord_p(S)\ge-3\min_i\ord_p(r_i).
\]
\end{proposition}

\begin{proof}
Fix $p\in C$ and retain the notation of
Proposition~\ref{prop:four-cases}:
\[
 m_i:=\ord_p(r_i),
 \qquad
 m:=\min_i m_i,
 \qquad
 I:=\{i:m_i=m\},
 \qquad
 J:=\{i:m_i>m\}.
\]
Set $\mu:=-3m$.  For any subset $B\subset\{1,\ldots,N\}$, write
\[
 S_B:=\sum_{i\in B}\frac{q_i}{r_i^3},
 \qquad
 T_B:=\sum_{i\in B}\frac{q_i\mathbf a_i}{r_i^3}.
\]
Every summand of $S_I$ has order $\mu$, and every component of every
summand of $T_I$ has order at least $\mu$.  Thus
\begin{equation}\label{eq:SI-TI-orders}
 \ord_p(S_I)\ge\mu,
 \qquad
 \ord_p(T_I)\ge\mu.
\end{equation}
Possible cancellation only increases these orders.

\subsubsection*{Case 1: a finite real point.}
Here $m=0$, so $\mu=0$, and $x$ is regular at $p$ with value
$\mathbf x_0\in A\subset\R^3$.  If $j\in J$, then
\[
 0=L_j(\Phi(p))
  =(\mathbf x_0-\mathbf a_j)\cdot
    (\mathbf x_0-\mathbf a_j)
  =|\mathbf x_0-\mathbf a_j|^2,
\]
so $\mathbf x_0=\mathbf a_j$.  Since the sites are distinct, $J$ contains
at most one index.

Suppose that $J=\{j\}$.  Then $m_j>0$, so $q_j/r_j$ has a pole at $p$.
For every $i\in I$, one has $m_i=0$, so $q_i/r_i$ is regular.  Thus the
single pole $q_j/r_j$ cannot cancel in the constant-potential identity
\eqref{eq:constant-potential}, a contradiction.  Hence $J$ is empty.
Therefore $S=S_I$ is regular and $\ord_p(S)\ge0=\mu$.

\subsubsection*{Case 2: a finite nonreal point.}
Again $m=0$ and $\mu=0$, but now
\[
 x(p)=z_0=u+iv,
 \qquad
 u,v\in\R^3,
 \qquad
 v\ne0.
\]
For $j\in J$, the equality $L_j(\Phi(p))=0$ becomes
$(z_0-\mathbf a_j)\cdot(z_0-\mathbf a_j)=0$.  Its imaginary part gives
$(\mathbf a_j-u)\cdot v=0$.  Consequently,
\[
 T_J\cdot v=(u\cdot v)S_J.
\]
Taking the scalar product with $v$ in the field identity
$x(S_I+S_J)=T_I+T_J$, cf. Equation~\eqref{eq:meromorphic-SxT}, gives
\begin{equation}\label{eq:finite-nonreal-SJ}
 (v\cdot x-u\cdot v)S_J=-v\cdot(xS_I-T_I).
\end{equation}
The coefficient on the left is holomorphic and nonzero at $p$, because
its value there is
\[
 v\cdot z_0-u\cdot v=i|v|^2\ne0.
\]
It is therefore invertible near $p$.  The right-hand side of
\eqref{eq:finite-nonreal-SJ} is regular by
\eqref{eq:SI-TI-orders}, since $x$ is regular at $p$.  Hence $S_J$ is
regular.  Both $S_I$ and $S_J$ are therefore regular, so
$\ord_p(S)\ge0=\mu$.

\subsubsection*{Case 3: the distinguished point at infinity.}
Here $\Phi(p)=[0:0:1]$.  For every $i$,
$L_i(0,0,1)=1$.  Proposition~\ref{prop:four-cases} therefore
gives $J=\varnothing$.  Thus $S=S_I$, and
\eqref{eq:SI-TI-orders} yields $\ord_p(S)\ge\mu$.

\subsubsection*{Case 4: any other point at infinity.}
Write
\begin{equation}\label{eq:nonvertex-infinity}
 \Phi(p)=[0:\Xi:\Theta],
 \qquad
 \Xi\ne0,
 \qquad
 \Xi\cdot\Xi=0.
\end{equation}
Let $\Xi=u+iv$ with $u,v\in\R^3$.  The equation $\Xi\cdot\Xi=0$ says
\begin{equation}\label{eq:isotropic-real-imag}
 |u|^2=|v|^2>0,
 \qquad
 u\cdot v=0.
\end{equation}
By~\eqref{eq:isotropic-real-imag}, $u$ and $v$ are linearly independent, and
\[
 \mathbf d:=u\times v
\]
is a nonzero real vector with
$\mathbf d\cdot\mathbf d=|\mathbf d|^2>0$.

If $J$ is empty, \eqref{eq:SI-TI-orders} already proves the desired
estimate.  Assume therefore that $J\ne\varnothing$.  For $j,j'\in J$, the
relations
$L_j(\Phi(p))=L_{j'}(\Phi(p))=0$ (cf. Equation~\eqref{eq:order-formulas}) give
\[
 (\mathbf a_j-\mathbf a_{j'})\cdot u=0,
 \qquad
 (\mathbf a_j-\mathbf a_{j'})\cdot v=0.
\]
Thus every difference $\mathbf a_j-\mathbf a_{j'}$ is parallel to
$\mathbf d$.  The sites with indices in $J$ lie on a real affine line, so
we may write
\begin{equation}\label{eq:J-line}
 \mathbf a_j=\mathbf a_*+\tau_j\mathbf d,
 \qquad
 \tau_j\in\R,
 \qquad
 j\in J, \qquad \mathbf a_*\in \mathbb R^3.
\end{equation}
If $J=\{j\}$, then $\tau_j=0$ and the next argument also applies. With the parametrization~\eqref{eq:J-line}, set
\[
 R_J:=\sum_{j\in J}\frac{\tau_jq_j}{r_j^3}.
\]
Then
\begin{equation}\label{eq:TJ-decomposition}
 T_J=\mathbf a_*S_J+\mathbf dR_J.
\end{equation}
Let
\[
 \sigma:=\ord_p(S_J),
 \qquad
 \varrho:=\ord_p(R_J),
\]
with $\ord_p(0)=+\infty$.  We claim that
\begin{equation}\label{eq:sigma-rho-bound}
 \sigma\ge\mu,
 \qquad
 \varrho\ge\mu.
\end{equation}
Suppose otherwise that either $\sigma$ or $\varrho$ are $<\mu$.  There are two possible orderings.

If $\sigma\le\varrho$, then necessarily $\sigma<\mu$.  Since
$S=S_I+S_J$ and $\ord_p(S_I)\ge\mu$, we have $\ord_p(S)=\sigma$.
Moreover, \eqref{eq:SI-TI-orders} and
\eqref{eq:TJ-decomposition} give $\ord_p(T)\ge\sigma$ (of course, $T=T_I+T_J$).  The identity
$x=T/S$ in $\mathcal M(C)$ then shows that every component of $x$ is
regular at $p$.  This contradicts~\eqref{eq:nonvertex-infinity}.

It remains to consider $\varrho<\sigma$.  In this case
$\varrho<\mu$.  In the decomposition
\[
 T=T_I+\mathbf a_*S_J+\mathbf dR_J,
\]
the last term is the unique term of order $\varrho$.  Hence $T$ has
leading direction $\mathbf d$.  On the other hand, both $S_I$ and $S_J$
have order strictly greater than $\varrho$, so, because $S$ is not
identically zero,
\[
 \delta:=\ord_p(S)>\varrho.
\]
Next, let us consider a local coordinate $t$ and an adapted Laurent expansion (recall that $C$ is a compact Riemann surface). Writing the leading terms of $T$ and $S$ in the local coordinate $t$, and using that $x=T/S$, we obtain
\[
 x=t^{-h}(\alpha\mathbf d+O(t)),
 \qquad
 h:=\delta-\varrho>0,
 \qquad
 \alpha\in\C^*.
\]
Since $\mathbf d\cdot\mathbf d>0$,
Lemma~\ref{lem:nonisotropic-leading} gives $m_i=-h$ for every $i$.  All
$m_i$ are therefore equal, so $J$ is empty, a contradiction.

Both alternatives lead to contradictions, so
\eqref{eq:sigma-rho-bound} holds.  Combining it with
\eqref{eq:SI-TI-orders} gives
\[
 \ord_p(S)=\ord_p(S_I+S_J)\ge\mu=-3m.
\]
This completes the fourth case and the proof.
\end{proof}

\begin{remark}\label{rem:no-dimension-reduction}
Proposition~\ref{prop:pole-estimate} is formulated in the ambient space
$\R^3$ and already includes collinear and coplanar configurations; no
separate reduction to their affine span is needed.
\end{remark}

Finally, we are ready to show the finiteness part of Theorem~\ref{thm:main}.

\begin{proposition}\label{prop:global-rigidity}
For arbitrary nonzero charges at pairwise distinct sites, the set $Z$ is
finite.
\end{proposition}

\begin{proof}
For $N=1$, the set is empty.  Suppose now that $N\ge2$ and that $Z$ is
infinite.  Lemma~\ref{lem:curve-reduction} produces a compact connected
Riemann surface $C$ and meromorphic functions $x$ and $r_i$ satisfying
\eqref{eq:distance-curve}, \eqref{eq:field-curve}, and
\eqref{eq:constant-potential}, with $x$ nonconstant and $S$ not identically
zero. 

At any
point of $C$, Proposition~\ref{prop:pole-estimate} gives
\[
 \ord_p(S)\ge-3\min_j\ord_p(r_j)
 \qquad(p\in C).
\]
For each $i$, define $K_i:=r_i^3S$.  At any $p\in C$,
\[
 \ord_p(K_i)
 =3\ord_p(r_i)+\ord_p(S)
 \ge3\ord_p(r_i)-3\min_j\ord_p(r_j)
 \ge0.
\]
Thus $K_i$ has no pole anywhere on $C$.  A meromorphic function without
poles on a compact connected Riemann surface is holomorphic and therefore
constant.  Neither $r_i$ nor $S$ is identically zero, so every $K_i$ is a
nonzero constant.

It follows that
\[
 \left(\frac{r_i}{r_j}\right)^3
 =\frac{K_i}{K_j}\in\C\backslash\{0\}
 \qquad(1\le i,j\le N).
\]
Since $C$ is connected, a meromorphic function whose cube is a fixed
nonzero constant is itself constant.  Thus
\[
 r_i=\kappa_i r_1,
 \qquad
 \kappa_i\in\C\backslash\{0\}.
\]
Consequently,
\[
 S=r_1^{-3}\alpha,
 \qquad
 T=r_1^{-3}\beta,
\]
where
\[
 \alpha:=\sum_{i=1}^Nq_i\kappa_i^{-3}\in\C,
 \qquad
 \beta:=\sum_{i=1}^Nq_i\kappa_i^{-3}\mathbf a_i\in\C^3.
\]
Because $S$ is not identically zero, $\alpha\ne0$.  The field identity
therefore yields $x=\beta/\alpha$, so $x$ is constant, which is a contradiction. The claim then follows.
\end{proof}

\section{The quantitative bound}
\label{sec:quantitative}

For the rest of the paper assume $N\ge2$; the case $N=1$ is trivial.  We encode each equilibrium
as a zero of a compact homogeneous system.  The ideas in
Section~\ref{sec:qualitative} allow us to show that the corresponding complex zero is
isolated.  We then compare with a generic system of the same bidegrees and
subtract the explicit contributions of the nonphysical zeros
at the charge sites and the distinguished point at infinity.

We use the ambient quadric
\[
 Y:=\{[X_0:X:X_4]\in\PP^4:X\cdot X=X_0X_4\}.
\]
It is obtained from $Y_A$ by omitting the affine-span equations in
\eqref{eq:YA}; in particular, it is smooth and has complex dimension three.
If the sites lie in a proper affine subspace, the system may have
additional nonphysical zeros.  They do not affect the argument, which uses
only the isolated zeros whose local contributions are explicitly controlled.
The distance forms from Section~\ref{sec:compactification} are
\[
 L_i(X_0,X,X_4)
 =X_4-2\mathbf a_i\cdot X+|\mathbf a_i|^2X_0,
 \qquad 1\le i\le N.
\]
On the affine chart $X_0=1$, this gives
$L_i=(x-\mathbf a_i)\cdot(x-\mathbf a_i)$.

Let $y=[y_1:\cdots:y_N]\in\PP^{N-1}$ and set
\[
 M:=Y\times\PP^{N-1}.
\]
Write $\mathbf a_i=(a_{i1},a_{i2},a_{i3})$.  On $M$ consider the
$N-1+3$ homogeneous equations
\begin{align}
 D_i&:=y_i^2L_i-y_N^2L_N=0,
 &&1\le i<N,\label{eq:projective-D}\\
 F_\alpha&:=\sum_{i=1}^Nq_i y_i^3
 (X_\alpha-a_{i\alpha}X_0)=0,
 &&1\le\alpha\le3.\label{eq:projective-F}
\end{align}
The $D_i$ have bidegree $(1,2)$ and the $F_\alpha$ have bidegree $(1,3)$.
Moreover,
\[
 \dim M=3+N-1=N+2,
\]
which is exactly the number of equations.

Let $\mathbf x\in Z$ and put
$\rho_i:=|\mathbf x-\mathbf a_i|$.  We define its physical lift as
\[
 \lambda(\mathbf x):=
 \left([1:\mathbf x:\mathbf x\cdot\mathbf x],
 [\rho_1^{-1}:\cdots:\rho_N^{-1}]\right)\in M.
\]
At this point $y_i^2L_i=\rho_i^{-2}\rho_i^2=1$, so all the $D_i$ vanish,
and the equations $F_\alpha=0$ are precisely the three components of
$E(\mathbf x)=0$.  The map $\lambda:Z\to M$ is injective because its first
component recovers $\mathbf x$.

For the B\'ezout count, real finiteness is not enough: a physical lift could
still lie on a complex curve of zeros.  The next proposition rules this out
using Proposition~\ref{prop:global-rigidity}.

\begin{proposition}\label{prop:complex-isolation}
For every $\mathbf x^0\in Z$, the point $\lambda(\mathbf x^0)$ is an
isolated complex solution of~\eqref{eq:projective-D}--\eqref{eq:projective-F}.
\end{proposition}

\begin{proof}
Suppose that $p:=\lambda(\mathbf x^0)$ were not isolated.  Some algebraic
component of the zero set through $p$ would have positive dimension.
Choose an irreducible projective curve $\Gamma\subset M$ through $p$ in
that component, and let
\[
 \nu_0:C_0\longrightarrow\Gamma
\]
be its normalization.  Since $\Gamma$ is projective, $C_0$ is a compact
connected Riemann surface.  Near $p$ we are in the chart $X_0y_N\ne0$, so
write
\[
 x=X/X_0,
 \qquad
 z_i=y_i/y_N\quad(i<N),
 \qquad
 z_N=1.
\]
These coordinate functions pull back to meromorphic functions on $C_0$.
Choose $p_0\in C_0$ with $\nu_0(p_0)=p$.  Since $z_i(p_0)\ne0$, none of
the $z_i$ is the zero meromorphic function.

Restricting the equations $D_i=0$ to the curve and pulling them back to
$C_0$ gives
\begin{equation}\label{eq:isolation-distance-ratios}
 z_i^2(x-\mathbf a_i)\cdot(x-\mathbf a_i)
 =(x-\mathbf a_N)\cdot(x-\mathbf a_N)
 \qquad(i<N).
\end{equation}
Set
\[
 f_N:=(x-\mathbf a_N)\cdot(x-\mathbf a_N).
\]
If $f_N$ has a meromorphic square root on $C_0$, choose one, call it
$r_N$, set $C:=C_0$, and let $\pi:C\to C_0$ be the identity.  Otherwise,
consider the two-sheeted cover of $C_0$ defined by
\[
 w^2=f_N,
\]
and let $C$ be its smooth compact model.  The projection
$\pi:C\to C_0$ is finite, and $w$ defines a meromorphic function on $C$;
set $r_N:=w$.  Since $f_N$ is not a square in $\mathcal M(C_0)$, the cover
is connected.  Pull back $x,z_1,\ldots,z_{N-1}$ to $C$ without changing
notation.  In the second case,
$f_N(p_0)=|\mathbf x^0-\mathbf a_N|^2>0$ shows that $\pi$ is unramified
over $p_0$.  Choose $\widetilde p_0\in C$ above $p_0$ and, after changing
the sign of $r_N$ if necessary, arrange that
\[
 r_N(\widetilde p_0)=|\mathbf x^0-\mathbf a_N|.
\]
Set
\[
 r_i:=r_N/z_i
 \qquad(i<N).
\]
No $r_i$ is identically zero, and
\begin{equation}\label{eq:isolation-distance-identities}
 r_i^2=(x-\mathbf a_i)\cdot(x-\mathbf a_i)
 \qquad(1\le i\le N).
\end{equation}

Restricting the equations $F_\alpha=0$ to the same dense chart gives
\[
 \sum_{i=1}^Nq_i z_i^3(x-\mathbf a_i)=0
\]
as a meromorphic identity on $C$.  Since $z_i=r_N/r_i$, division by
$r_N^3$ yields
\begin{equation}\label{eq:isolation-field-identity}
 \sum_{i=1}^Nq_i\frac{x-\mathbf a_i}{r_i^3}=0.
\end{equation}
We must also verify that $S:=\sum_iq_i/r_i^3$ is not the zero function.
Put $\Sigma:=\sum_iq_i z_i^3$.  At $\widetilde p_0$,
\[
 \Sigma(\widetilde p_0)
 =|\mathbf x^0-\mathbf a_N|^3S(\mathbf x^0)\ne0,
\]
and hence $S=r_N^{-3}\Sigma$ is not identically zero.

Finally, differentiating~\eqref{eq:isolation-distance-identities} gives
$r_i\,dr_i=(x-\mathbf a_i)\cdot dx$.  Equation
\eqref{eq:isolation-field-identity} then yields
\[
 d\!\left(\sum_{i=1}^N\frac{q_i}{r_i}\right)
 =-\sum_{i=1}^Nq_i\frac{(x-\mathbf a_i)\cdot dx}{r_i^3}=0.
\]
Thus the sum is constant on $C$.  This way, we have obtained a compact connected
Riemann surface $C$ and meromorphic functions $x$ and $r_i$ satisfying the Equations
\eqref{eq:distance-curve}, \eqref{eq:field-curve}, and
\eqref{eq:constant-potential} on $C$, with $S$ and $r_i$ not identically zero. It is then clear that we can apply the same proof as in
Proposition~\ref{prop:global-rigidity} to infer that the meromorphic function $x$ is constant.
Since $x(\widetilde p_0)=\mathbf x^0$, we have $x\equiv\mathbf x^0$, and
\eqref{eq:isolation-distance-ratios} gives
\[
 z_i^2=
 \frac{(\mathbf x^0-\mathbf a_N)\cdot
       (\mathbf x^0-\mathbf a_N)}
      {(\mathbf x^0-\mathbf a_i)\cdot
       (\mathbf x^0-\mathbf a_i)}
 \in\C\backslash\{0\}.
\]
Since $C$ is connected, every $z_i$ is constant.  The map
$\nu_0\circ\pi:C\to\Gamma$ is finite and surjective, but all of
its affine coordinate functions are constant on a dense open subset.  It
is therefore constant, contradicting that $\Gamma$ is a curve, and the proof is complete.
\end{proof}

We compare our system with a generic system of the same bidegrees.  The
required global count is the following.

\begin{lemma}\label{lem:global-bezout}
A generic system on $M=Y\times\PP^{N-1}$ consisting of $N-1$ equations of
bidegree $(1,2)$ and three equations of bidegree $(1,3)$ has
\begin{equation}\label{eq:TN}
 \alpha_N:=2\sum_{k=0}^{\min\{3,N-1\}}
 3^k2^{N-1-k}\binom{3}{k}\binom{N-1}{k}
 =2^{N-4}(9N^3+9N-2)
\end{equation}
simple zeros.  Moreover, under a sufficiently small perturbation, the
total local multiplicity inside any fixed isolating neighborhood is
unchanged.
\end{lemma}

\begin{proof}
We first show that a generic system has only simple zeros.  For each \(k\in\{2,3\}\), we claim that the family consisting of all equations of bidegree
\((1,k)\) has no common zero on \(M\).  Indeed, at any point
\((p,[y])\in M\), one can choose a linear form in the coordinates of
\(\PP^4\) that is nonzero at \(p\), and a degree-\(k\) homogeneous
polynomial in the coordinates of \(\PP^{N-1}\) that is nonzero at
\([y]\).  Their product is then an equation of bidegree \((1,k)\) that does
not vanish at \((p,[y])\). This is precisely the base-point-free condition required
by Bertini's theorem~\cite[Corollary~III.10.9]{Hartshorne1977}.

We can therefore apply Bertini's theorem one equation at a time.  Starting from the smooth
variety \(M\), a generic first equation defines a smooth hypersurface.
Restricting the remaining families to this hypersurface, a generic
second equation meets it transversely.  Repeating this argument for all
the equations, their generic common zero set is smooth of the expected
dimension. Accordingly, a generic tuple of all the equations cuts
\(M\) transversely.  Since
\[
 \dim M=\dim Y+\dim\PP^{N-1}=3+(N-1)=N+2
\]
and the system contains
\[
 (N-1)+3=N+2
\]
equations, its common zero set is zero-dimensional and reduced.  Thus it
is finite and all its zeros are simple.  Denote this zero set by
\(Z_{\mathrm{gen}}\).

We now count these zeros.  The multihomogeneous B\'ezout rule used here
is the direct analogue of the classical one.  In the classical B\'ezout
theorem on a single projective space, a hypersurface of degree \(d\)
contributes \(d h\), where \(h\) denotes a hyperplane class (a topological condition in algebraic geometry, see below), and
multiplying these contributions gives the product of the degrees.  Here
there are two projective factors, so the degrees in the two factors must
be recorded separately.

Accordingly, let \(u\) and \(v\) denote hyperplane classes in the
first and second factors.  More precisely,
\[
 u=c_1\bigl(\pi_1^*\mathcal O_Y(1)\bigr),
 \qquad
 v=c_1\bigl(\pi_2^*\mathcal O_{\PP^{N-1}}(1)\bigr),
\]
where \(c_1\) denotes the first Chern class, and $\pi_1:M\to Y$, $\pi_2:M\to \mathbb P^{N-1}$ are the canonical projectors; in this setting, it is simply
the divisor class of a hyperplane section.  An equation of bidegree
\((a,b)\) has class
\[
 au+bv.
\]
Thus each of the \(N-1\) equations of bidegree \((1,2)\) contributes
\(u+2v\), and each of the three equations of bidegree \((1,3)\)
contributes \(u+3v\).

Expanding the product of all the contributions
\[
 (u+2v)^{N-1}(u+3v)^3
\]
amounts to distributing the intersection conditions between the two
projective factors.  To obtain isolated points, exactly three
hyperplane conditions must be imposed on the three-dimensional variety
\(Y\), and exactly \(N-1\) on \(\PP^{N-1}\).  Hence only the coefficient
of \(u^3v^{N-1}\) contributes to the count.  Moreover, three general
hyperplanes cut the quadric \(Y\subset\PP^4\) in two points, while
\(N-1\) general hyperplanes in \(\PP^{N-1}\) have one common point.
The multihomogeneous B\'ezout theorem
\cite[Example~8.4.2]{Fulton1998} therefore gives
\[
 \#Z_{\mathrm{gen}}
 =
 2[u^3v^{N-1}](u+2v)^{N-1}(u+3v)^3,
\]
where \([u^av^b]P\) denotes the coefficient of \(u^av^b\) in \(P\).
This is the analogue, for two groups of homogeneous variables, of
multiplying the degrees in the classical B\'ezout theorem.

It remains to extract this coefficient.  Let \(k\) be the number of
times that the term \(3v\) is selected from the three factors
\(u+3v\).  This contributes
\[
 3^k\binom{3}{k}u^{3-k}v^k.
\]
To obtain \(u^3v^{N-1}\), the other factor must then contribute
\(u^kv^{N-1-k}\).  This term is
\[
 2^{N-1-k}\binom{N-1}{k}u^kv^{N-1-k}.
\]
Consequently,
\[
 \#Z_{\mathrm{gen}}
 =
 2\sum_{k=0}^{\min\{3,N-1\}}
 3^k2^{N-1-k}\binom{3}{k}\binom{N-1}{k}.
\]
Using the convention \(\binom{N-1}{k}=0\) for \(k>N-1\), this can be
written as
\begin{align*}
 \#Z_{\mathrm{gen}}
 &=2^{N-4}\bigl(
 16+72(N-1)+54(N-1)(N-2)\\
 &\hspace{35mm}
 +9(N-1)(N-2)(N-3)
 \bigr)\\
 &=2^{N-4}(9N^3+9N-2)
 =\alpha_N,
\end{align*}
which is Equation~\eqref{eq:TN} in the statement of the lemma.

Finally, consider an arbitrary system of the same bidegrees and an
isolating neighborhood \(U\), so that the system has only isolated zeros
in \(U\) and no zero on \(\partial U\).  Since the system does not vanish
on \(\partial U\), every sufficiently small perturbation is also
nonvanishing there, and no zero can enter or leave \(U\).  The
conservation-of-number principle states that the sum of the local
multiplicities of the zeros in \(U\) is therefore unchanged.  For a
generic perturbation all these zeros are simple, so this sum is exactly
the number of perturbed zeros lying in \(U\).

The perturbation may be chosen simultaneously generic and sufficiently
small for any fixed finite collection of pairwise disjoint isolating
neighborhoods.  Since it has exactly \(\alpha_N\) zeros on all of \(M\),
the sum of the local multiplicities contained in those neighborhoods is
at most \(\alpha_N\); see \cite[Section~10.2]{Fulton1998}.
\end{proof}

Furthermore, the system of equations~\eqref{eq:projective-D} and~\eqref{eq:projective-F} on $M$ has two families of zeros that do not represent
physical equilibria.  The first lies over the charge sites, which are
excluded from the physical domain.  For $1\le j\le N$, let $[e_j]$
denote the $j$th
coordinate point of $\PP^{N-1}$ and set
\[
 p_j:=\bigl([1:\mathbf a_j:|\mathbf a_j|^2],[e_j]\bigr)\in M.
\]
At $p_j$ every field equation vanishes, and so does every product
$y_i^2L_i$; hence $p_j$ is a zero of the system.

\begin{proposition}\label{prop:site-multiplicity}
Each $p_j$ is an isolated zero and has local multiplicity $2^{N-1}$.
\end{proposition}

\begin{proof}
Work in the charts $X_0=1$ and $y_j=1$, and write
\[
 w=x-\mathbf a_j,
 \qquad
 u_k=y_k\quad(k\ne j),
 \qquad
 \mathbf b_k=\mathbf a_j-\mathbf a_k.
\]
Instead of comparing every distance with the $N$th one, compare it with
the $j$th one.  This is an invertible linear combination of the $D_i$
and gives the equivalent equations
\begin{equation}\label{eq:Delta-k}
 \Delta_k:=u_k^2(w+\mathbf b_k)\cdot(w+\mathbf b_k)-w\cdot w=0,
 \qquad k\ne j.
\end{equation}
The field equations become
\[
 G(w,u):=q_jw+
 \sum_{k\ne j}q_ku_k^3(w+\mathbf b_k)=0.
\]
Their derivative with respect to $w$ at the origin is $q_jI_3$, so the
holomorphic implicit-function theorem solves them uniquely as
\[
 w=\phi(u):=
 -\frac{\sum_{k\ne j}q_ku_k^3\mathbf b_k}
 {q_j+\sum_{k\ne j}q_ku_k^3}
 =O(|u|^3).
\]
Substituting in~\eqref{eq:Delta-k} gives
\[
 h_k(u):=\Delta_k(\phi(u),u)
 =|\mathbf b_k|^2u_k^2+O(|u|^5),
 \qquad k\ne j.
\]
Here $|\mathbf b_k|^2>0$ because the sites are distinct.

For $0<t\le1$, set the homotopy
\[
 h^{(t)}(u):=t^{-2}h(tu),
\]
and extend continuously to
\[
 h^{(0)}(u)=\bigl(|\mathbf b_k|^2u_k^2\bigr)_{k\ne j}.
\]
For small $|u|$, uniformly in $0\le t\le1$,
\[
 \|h^{(t)}(u)-h^{(0)}(u)\|\le Ct^3|u|^5,
 \qquad
 \|h^{(0)}(u)\|\ge c|u|^2.
\]
Thus the homotopy has no zero on a sufficiently small sphere, and
conservation of number shows that $h$ and $h^{(0)}$ have the same local
multiplicity.  The diagonal system $u_k^2=0$ has multiplicity $2^{N-1}$:
a small perturbation gives two roots in each coordinate, independently.
The map $(w,u)\mapsto(G(w,u),u)$ is locally biholomorphic, so eliminating
$w$ does not change the multiplicity of the full system.
\end{proof}

The second family lies over the distinguished point at infinity.  Let
\[
 v_\infty:=[0:0:1]\in Y
\]
be the distinguished point at infinity from
Section~\ref{sec:compactification}.  At $v_\infty$ one has
$L_i(v_\infty)=1$, while
all the field equations vanish.  The distance equations reduce to
$y_i^2=y_N^2$, and hence give exactly the $2^{N-1}$ points
\begin{equation}\label{eq:infinity-roots}
 p_\varepsilon:=
 \bigl(v_\infty,[\varepsilon_1:\cdots:\varepsilon_{N-1}:1]\bigr),
 \qquad
 \varepsilon_i\in\{\pm1\}.
\end{equation}
Some may be degenerate.  To make them simple, change only the field
equations by replacing the coefficient $q_N$ there with $q_N+\delta$:
\begin{equation}\label{eq:structured-perturbation}
 F_{\alpha,\delta}:=
 F_\alpha+\delta y_N^3(X_\alpha-a_{N\alpha}X_0),
 \qquad 1\le\alpha\le3.
\end{equation}
This is still a system of the same bidegrees; denote it by $s_\delta$.

Use the chart $X_4=1$, where the quadric is parametrized by
$X=z$ and $X_0=z\cdot z$, and set $y_N=1$.  With the variables ordered as
$(y,z)$, the Jacobian at $p_\varepsilon$ is block upper triangular.  Its
diagonal blocks are
\[
 2\operatorname{diag}(\varepsilon_1,\ldots,\varepsilon_{N-1})
 \quad\text{and}\quad
 (C_\varepsilon+\delta)I_3,
 \qquad
 C_\varepsilon:=q_N+\sum_{i<N}q_i\varepsilon_i.
\]
Consequently,
\[
 \det ds_\delta(p_\varepsilon)
 =2^{N-1}(\varepsilon_1\cdots\varepsilon_{N-1})
  (C_\varepsilon+\delta)^3.
\]
Choose an arbitrarily small $\delta\in\C$ such that
\[
 \delta\ne-C_\varepsilon\quad\text{for every }\varepsilon,
 \qquad
 \delta\ne-q_N.
\]
Then all the points~\eqref{eq:infinity-roots} are simple zeros of
$s_\delta$.

The same perturbation preserves every charge-site zero and its
multiplicity.  In the coordinates used in
Proposition~\ref{prop:site-multiplicity}, it only replaces $q_N$ by
$q_N+\delta$ in the field equation.  If $j<N$, the derivative with respect
to $w$ at the origin remains $q_jI_3$; if $j=N$, it becomes
$(q_N+\delta)I_3$.  In both cases it is invertible, the solution still
satisfies $\phi_\delta(u)=O(|u|^3)$, and the reduced distance equations
remain
\[
 |\mathbf b_k|^2u_k^2+O(|u|^5).
\]
The homotopy in the proof of Proposition~\ref{prop:site-multiplicity} is
uniform for small $\delta$, so every $p_j$ still contributes $2^{N-1}$.

These preparations allow us to obtain the final count. The following proposition then completes the proof of Theorem~\ref{thm:main}.

\begin{proposition}\label{prop:quantitative-bound}
The cardinality of the nonexceptional equilibrium set is bounded from above as
\begin{equation}\label{eq:final-bound}
 \#Z\le
 2\sum_{k=0}^{\min\{3,N-1\}}
 3^k2^{N-1-k}\binom{3}{k}\binom{N-1}{k}
 -(N+1)2^{N-1}
 =2^{N-4}(N-1)(9N^2+9N+10).
\end{equation}
\end{proposition}

\begin{proof}
By Proposition~\ref{prop:global-rigidity}, the set $Z$ is finite.
By Proposition~\ref{prop:complex-isolation}, each physical lift
$\lambda(\mathbf x)$ is an isolated complex zero of the original system;
let its local multiplicity be $\mu_{\mathbf x}\ge1$.  Choose pairwise
disjoint small neighborhoods of all physical lifts and all charge-site
zeros, with no zero on their boundaries.

For sufficiently small admissible $\delta$, the perturbation
\eqref{eq:structured-perturbation} remains nonzero on all these boundaries.
Conservation of number therefore keeps total multiplicity
$\mu_{\mathbf x}$ in each physical neighborhood and $2^{N-1}$ in each site
neighborhood.  Add small disjoint neighborhoods of the $2^{N-1}$ simple
zeros at infinity.  Finally, perturb $s_\delta$ arbitrarily slightly in the
full vector spaces of sections of bidegrees $(1,2)$ and $(1,3)$.  No zero
crosses any chosen boundary, so the perturbed system contains at least
\[
 \sum_{\mathbf x\in Z}\mu_{\mathbf x}
 +N2^{N-1}+2^{N-1}
 \ge \#Z+(N+1)2^{N-1}
\]
simple zeros in these neighborhoods.  Lemma~\ref{lem:global-bezout} gives
\[
 \alpha_N\ge \#Z+(N+1)2^{N-1}.
\]
Combining this with~\eqref{eq:TN} proves~\eqref{eq:final-bound}.  It is clear that other
components of the zero set defined by Equations~\eqref{eq:projective-D} and~\eqref{eq:projective-F}, away from the selected neighborhoods,
do not affect the argument.
\end{proof}

\section*{Acknowledgements}

This work has received funding from the European Research Council (ERC)
under the European Union's Horizon 2020 research and innovation programme
through grant agreement 862342 (A.E.).  The authors are also partially
supported by the grant PID2022-
136795NB-I00 of the Spanish Science Agency and
the ICMAT--Severo Ochoa grant CEX2023-001347-S.

\appendix

\section{Mixed-sign fields with curves of zeros}
\label{sec:mixed-sign}

The following elementary examples are well known, and we only include them
to show that the restriction to the nonexceptional set $S\ne0$, or to
charges of one sign, is necessary.

First, place four charges at the vertices of a square, with
\[
 \mathbf a_1=(1,1,0),\quad
 \mathbf a_2=(-1,1,0),\quad
 \mathbf a_3=(-1,-1,0),\quad
 \mathbf a_4=(1,-1,0),
\]
and
\[
 q_1=q_3=1,
 \qquad
 q_2=q_4=-1.
\]
Then both the electric field $E$ and the function $S$ vanish identically on the line
\[
 \{\mathbf x=(0,0,t):t\in\R\}.
\]

Likewise, place three charges on a line, with
\[
 \mathbf a_1=(0,0,-4),\quad
 \mathbf a_2=(0,0,0),\quad
 \mathbf a_3=(0,0,4),
\]
and
\[
 q_1=q_3=1,
 \qquad
 q_2=-\frac{54}{125}.
\]
Then both $E$ and $S$ vanish identically on the circle
\[
 \{\mathbf x=(3\cos t,3\sin t,0):t\in\R\}.
\]

\end{document}